\documentclass[11pt]{amsart}
\usepackage{geometry}               
\usepackage[parfill]{parskip}   
\usepackage{graphicx,color}
\usepackage{amssymb}
\usepackage{epstopdf}
\usepackage{eucal}
\theoremstyle{theorem}
\newtheorem{theorem}{Theorem}
\newtheorem{lemma}{Lemma}

\newtheorem{proposition}{Proposition}
\theoremstyle{definition}
\newtheorem{definition}{Definition}

\title{Coxeter groups and meridional rank of links}
\author{Sebastian Baader, Ryan Blair and Alexandra Kjuchukova}

\begin{document}

\begin{abstract} We prove the meridional rank conjecture for twisted links and arborescent links associated to bipartite trees with even weights. These links are substantial generalizations of pretzels and two-bridge links, respectively. Lower bounds on meridional rank are obtained via Coxeter quotients of the groups of link complements. Matching upper bounds on bridge number are found using the Wirtinger numbers of link diagrams, a combinatorial tool developed by the authors. 
\end{abstract}
\maketitle

\section{Introduction}
The meridional rank $\mu$ of a link $L$ in $S^3$ is the minimal number of meridians of $L$ needed to  generate  $\pi_1(S^3 \setminus L)$. It is an immediate consequence of the Wirtinger presentation for $\pi_1(S^3 \setminus L)$ in a suitable diagram
that $\mu(L)$ is bounded above by the bridge number $\beta(L)$.
The meridional rank conjecture asks whether the equality $\mu(L)=\beta(L)$ holds. This question originates with Cappell and Shaneson's work on the Smith Conjecture~\cite{CS1978note} and is given as problem 1.11 in~\cite{kirby1995problems}. 

Boileau and Zimmermann~\cite{boileau1989orbifold} showed that  $\mu=2$ implies $\beta=2$. The equality $\beta=\mu$ has been established in various special cases, such as Montesinos links~\cite{boileau1985nombre},  torus links~\cite{RZ87}, and others whose complements satisfy certain geometric conditions~\cite{LM93, CH14, boileau2017meridionalrank, boileau2017meridional, baader2017symmetric}. 

We prove the meridional rank conjecture for two new classes, twisted links and arborescent links associated with bipartite trees with even weights. We also explicitly compute the bridge numbers of all links in these classes.  

To define twisted links, let $D$ be a diagram of a link $L$, admitting no reducing Reidemeister moves of type I and II, and let $F$ be one of the two surfaces with boundary $L$ obtained from a checkerboard coloring of the regions in the plane determined by $D$. We regard the surface $F$ as a union of disks and twisted bands, whose combinatorics we store in a plane graph $\Gamma \subset \mathbb{R}^2$ with weighted edges. We say the surface $F$ is twisted if every band has at least one full twist, and if the plane dual graph $\Gamma^\ast \subset \mathbb{R}^2$ of $\Gamma$ has no multiple edges. 
A link is twisted if it admits a diagram which determines a twisted surface via a checkerboard coloring. Figure~\ref{pretzelknot}, a pretzel~knot,~is~an~example~of~a~twisted~diagram.

\begin{theorem}\label{twisted}
	The meridional rank conjecture holds for twisted links. The bridge number of a twisted link is equal to the number of planar regions in the complement of the projection of a twisted surface. 
\end{theorem}

The class of arborescent links generalizes both two-bridge links and Montesinos links. They are defined by plumbing twisted bands in a tree-like pattern. More precisely, an arborescent link is associated to a plane tree $T \subset \mathbb{R}^2$ with weighted vertices. The vertices are in one-to-one correspondence with embedded annuli; their integer weights indicate the number of half-twists of the corresponding annuli. A precise definition of how the annuli are to be plumbed together along the edges of~$T$ can be found in~\cite{gabaigenera}. We will only consider trees with even non-zero weights, a condition which implies that all the bands involved are orientable, and that their union forms a minimal genus Seifert surface of the corresponding link, see again~\cite{gabaigenera}.
For technical reasons, we will also restrict the class of trees. A (plane) tree is called bipartite, if all the vertices of valency at least three carry the same color with respect to any of the two bipartite colorings of that tree. An example of a plane bipartite tree with even weights is shown in Figure~\ref{tree}. The corresponding arborescent link is shown in Figure~\ref{arborescentlink} (with some additional labels for later use). The class of arborescent links associated with even weight bipartite trees contains all two-bridge links. Indeed, the latter correspond to even weight trees ``without branches", i.e. to trees homeomorphic to an interval, see~\cite{conway1970enumeration}. On the other extreme, the class of arborescent links associated with even weight bipartite trees also contains the class of slalom divide links defined by A'Campo~\cite{a1998planar}. In fact, these links are obtained by plumbing positive Hopf bands along bipartite trees. In our setting, this means that all weights are two. This follows from the visualisation algorithms for divide links described in~\cite{van2000link} and~\cite{hirasawa2002visualization}.

\begin{figure}[h]
\begin{center}
\includegraphics[scale=1.8]{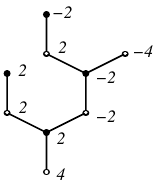}
\caption{Bipartite tree with even weights. \label{tree}}
\end{center}
\end{figure}

\begin{theorem}\label{arbs}
	The meridional rank conjecture holds for arborescent links associated with bipartite trees with non-zero even weights. The bridge number of such a link is equal to the number of leaves of the underlying tree.
\end{theorem}

\begin{figure}[h]
\begin{center}
\includegraphics[scale=0.7]{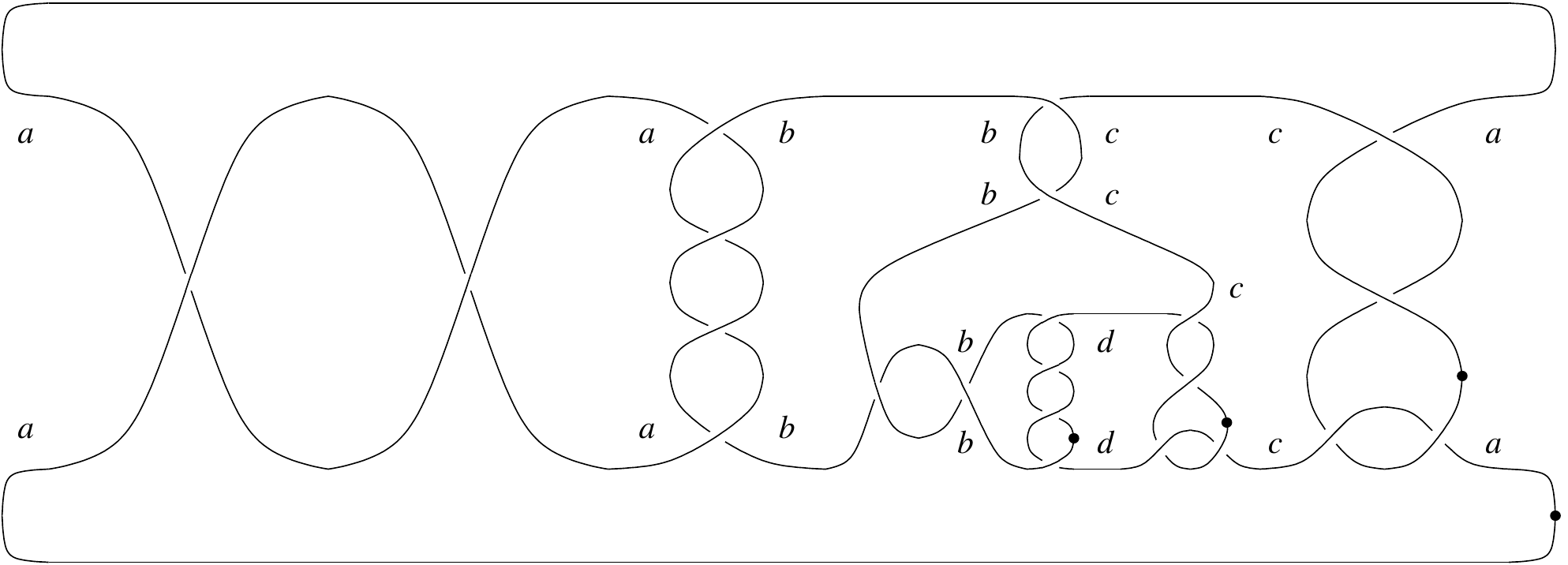}
\caption{Arborescent knot. \label{arborescentlink}}
\end{center}
\end{figure}

We prove Theorems~\ref{twisted} and~\ref{arbs} by obtaining an upper bound on the bridge number $\beta(L)$ 
and a matching lower bound on the meridional rank $\mu(L)$, from a suitable diagram. The lower bound on $\mu$ arises from a Coxeter quotient of $\pi_1(S^3 \setminus L)$ mapping meridians to reflections; see Proposition~\ref{lowerbound}.  
The upper bound on $\beta$ comes from the Wirtinger number $\omega$ of a link diagram $D$; see Section~\ref{upper}. 
The bridge number $\beta(L)$ equals the minimum value of $\omega(D)$  over all diagrams $D$ of $L$~\cite[Theorem 1.3]{blair2020wirtinger}. 
As we will see, if a link $L$ admits a diagram $D$ and a Coxeter quotient of rank equal to $\omega(D)$, the meridional rank conjecture holds for $L$. Our approach was inspired by a method for obtaining Coxeter and Artin quotients from knot diagrams introduced in~\cite{brunner1992geometric}.

Besides establishing the meridional rank conjecture for new classes of links, our technique also recovers the result for pretzel links and, more generally, Montesinos links, in a new way. We also remark that, for knots whose meridional rank is detected via Coxeter quotients, meridional rank is seen to satisfy Schubert additivity under connected sum without relying on equality with bridge number. Additivity of meridional rank under connected sum, which is implied by the meridional rank conjecture, is an interesting open question in its own right.

\section{Lower bounds on meridional rank}\label{lower}

The rank of a group~$G$ is the minimal cardinality among all generating sets of~$G$. The meridional rank of a link $L$ is clearly bounded below by the rank of its fundamental group, thus by the rank of any quotient of the latter. However, this is not an effective bound, since there is an abundance of links with rank two fundamental groups and arbitrarily high meridional rank, for example torus links. This fact carries over to a variety of groups with a geometric flavour: mapping class groups, symmetric groups and finite irreducible Coxeter groups have rank two, but they are typically not generated by a small number of standard generators, such as Dehn twists, transpositions and reflections, respectively. We should thus expect much better lower bounds on the meridional rank of links by considering quotients with a distinguished conjugacy class (on which the meridians of the link are to be mapped), which does not admit a small number of generators. We will apply this method to the class of Coxeter groups, and the conjugacy class of reflections. Recall that the Coxeter group $C(\Delta)$ associated with a finite simple graph $\Delta$ with weighted edges is the group whose generators are in bijection with the vertices of $\Delta$, subject to the following two types of relations:
\begin{enumerate}
\item $s^2=1$ for all generators $s$,
\item $(st)^k=1$, for all pairs of generators $s,t$ connected by an edge of weight $k \in \mathbb{N}$.
\end{enumerate}
Throughout this paper, we assume all edge weights to be at least two.
Elements of a Coxeter group $G=C(\Gamma)$ conjugate to any of the generators are called reflections. 
We refer to the number of vertices of the graph $\Gamma$ as the rank $r(C(\Gamma))$. It equals the minimal number of reflections needed to generate $C(\Gamma),$ see for example Lemma~2.1 in~\cite{felikson2010reflection}. Note that there exist graphs $\Gamma_1$ and $\Gamma_2$ with different numbers of vertices and such that the groups $C(\Gamma_1)$ and  $C(\Gamma_2)$ are isomorphic. In particular, the notions of reflection and rank of a Coxeter group depend on a choice of generating set. We thus obtain the following lower bound on the meridional rank of links.

\begin{proposition} \label{lowerbound}
Let $L$ be a link whose fundamental group surjects onto a Coxeter group $C(\Gamma)$, so that all meridians are mapped to reflections in $C(\Gamma)$. Then
$$\mu(L) \geq r(C(\Gamma)).$$ 
\end{proposition}

Throughout the paper, we will consider Coxeter groups that arise as quotients of a link group by sending all meridians to reflections. We refer to such groups as Coxeter quotients of the corresponding link. They were introduced by Brunner, in the guise of Artin quotients~\cite{brunner1992geometric}; his construction was the starting point for our work. 

The easiest examples of links admitting non-trivial Coxeter quotients are torus links of type $T(2,\pm n)$, i.e. closures of the 2-braid $\sigma_1^{\pm n} \in B_2$, where $n \geq 2$ is a natural number. We claim that the fundamental group of $T(2,\pm n)$ surjects to the rank two Coxeter group $D_n$ generated by two reflections $a,b$ satisfying the relation
$$(ab)^n=1.$$
The following diagram illustrates a consistent way of mapping the meridians of the diagram associated with
 the closure of the braid $\sigma_1^n \in B_2$ to reflections of $D_n$, for $n=3$.

\begin{figure}[h]
\begin{center}
\includegraphics[scale=1.4]{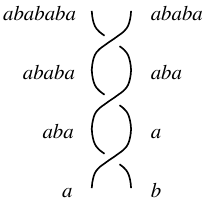}
\caption{Coloring a two-braid with reflections. \label{coloredbraid}}
\end{center}
\end{figure}

Here the orientation of the meridians does not matter since these are all mapped to reflections, which have order two. The labeling of the arcs is compatible with the Wirtinger conjugation relation at each crossing:
$$((ab)^ka)((ab)^{k-1}a)((ab)^ka)^{-1}=(ab)^{k+1}a.$$
Moreover, the relation $(ab)^n=1$ insures that the meridians at the top of the braid are mapped again to $(ab)^na=a$ and $(ab)^nb=b$, respectively. 
Proposition~\ref{lowerbound} implies that the meridional rank of two-bridge torus links is at least two, hence exactly two: $\mu(L)=\beta(L)=2$.
These examples are part of two larger families, pretzel links and two-bridge links, whose meridional rank is detected by the rank of suitable Coxeter quotients.

A \emph{twist region} is a maximal string of bigon regions in the knot projection, arranged end–to–end at their vertices. Pretzel links are defined via certain diagrams $P(a_1,a_2,\ldots,a_k)$ with $k \geq 3$ vertical twist regions. The coefficients $a_i \in \mathbb{Z}$ encode the number of crossings in each twist region, and their signs. The 
 convention can be deduced
 from Figure~\ref{pretzelknot}, which shows the pretzel knot $P(-2,3,5)$.

\begin{figure}[h]
\begin{center}
\includegraphics[scale=1.4]{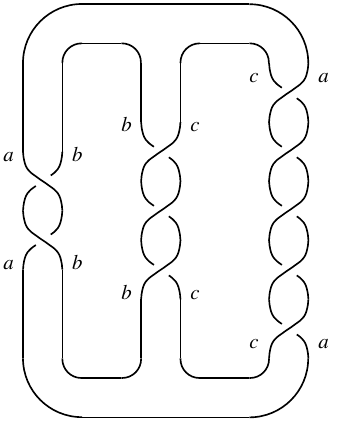}
\caption{Pretzel knot $P(-2,3,5)$. \label{pretzelknot}}
\end{center}
\end{figure}

Borrowing from the above discussion on two-bridge torus links, we deduce that pretzel links of type $P(a_1,a_2,\ldots,a_k)$ with all $|a_i| \geq 2$ admit a rank~$k$ Coxeter group quotient: $C(\Delta(a_1,a_2,\ldots,a_k))$, where $\Delta(a_1,a_2,\ldots,a_k)$ denotes a cycle with~$k$ vertices, whose edges are labeled $|a_1|,\ldots,|a_k|$, in a cyclic way. This can be seen in Figure~\ref{pretzelknot}, where a certain generating set of meridians of the pretzel knot $P(-2,3,5)$ is mapped to the generators $a,b,c$ of a rank three Coxeter group, satisfying the relations $(ab)^2=1$, $(bc)^3=1$, $(ac)^5=1$. The lower bound on the meridional rank from Proposition~\ref{lowerbound}, $\mu(L) \geq k$, matches the bridge number again. Indeed, the standard diagram of the pretzel link $P(a_1,a_2,\ldots,a_k)$ has exactly $k$ local maxima. Therefore, we have just reproved the meridional rank conjecture for pretzel links with every $|a_i|\geq 2$. The original proof by Boileau and Zieschang was based on 2-fold branched coverings~\cite{boileau1985nombre}.

We now briefly turn to two-bridge knots, to which our method also applies. By using the calculus of continued fractions, one can see that two-bridge knots are determined by a rational number $\alpha/\beta$ with relatively prime odd integers $\alpha, \beta$ such that
$-\alpha<\beta<\alpha$, see~\cite{conway1970enumeration}. The fundamental group of the knot $L(\alpha/\beta)$ admits a presentation with two generators $x,y$ and one relation of the form
$wx=yw$, where
$$w=x^*y^* \cdots x^*y^*$$
is a word of even length $\alpha-1$ and each star stands for a sign $\pm 1$, determined by the fraction $\alpha/\beta$. This is taken from~\cite{kitano2017note}. Setting $x^2=y^2=1$ reduces the relation $wx=yw$ to the Coxeter relation $(xy)^{\alpha}=1$. The special case $\alpha/\beta=n/1$ for odd $n$ corresponds to the torus knots of type $T(2,n)$ discussed previously. We conclude that all non-trivial two-bridge knots admit a Coxeter quotient of rank two. We can visualize these Coxeter quotients as a labeling of the strands of a diagram of a two-bridge knot by elements in the corresponding Coxeter group. Since every rational tangle diagram can be completed to a two-bridge knot by attaching a trivial tangle, then every rational tangle diagram has a labeling by elements of a rank two Coxeter group. Moreover, the strands of the rational tangle that are incident to the boundary of the tangle receive labels in the generating set $\{x,y\}$ and the labels of these strands are cyclically ordered around the boundary of the tangle according to the pattern $xxyy$. We call such a labeling a \emph{rank 2 (Coxeter) labeling} of a rational tangle.

It is known 
that the only links with meridional rank two are two-bridge links~\cite{boileau1989orbifold}. We do not know whether this can be seen by considering the maximal rank among all Coxeter quotients of a link. In fact, we do not even know whether all non-trivial knots admit non-cyclic Coxeter quotients.

\section{Upper bounds on bridge number}\label{upper}

Let $D$ be a diagram of a link $L$. To obtain the desired upper bound on the bridge number of $L$, we will use the Wirtinger number, $\omega(D)$, introduced in~\cite{blair2020wirtinger}. The Wirtinger number is an integer associated to a knot diagram. It can be determined via a combinatorial procedure for coloring the diagram, as recalled below. It formalizes the idea of finding the minimal number of Wirtinger generators in $D$ sufficient to generate the group $\pi_1(S^3-L)$ by only using ``iterated Wirtinger relations" in $D$.

Denote by $n$ the crossing number of $D$ and think of $D$ as the union of $n$ strands, or closed arcs in the plane. Two strands are adjacent if they are the under-strands at some crossing in $D$. Denote the set of strands in $D$ by $S(D)$. We say that $D$ is partially colored if we have fixed a function $f: S(D) \to \{0, 1\}$ such that $A:=\{s|f(s)=1\}\neq\emptyset$.  Given such a function $f$, we refer to the elements of $S(D)$ on which $f$ evaluates to 1 as the colored strands of $D$, and we refer to $A$ as a partial coloring of $D$.  Given two partial colorings $A_1$ and $A_2$ of the same diagram $D$, we say $A_2$ can be obtained from  $A_1$ via a coloring move on $D$, denoted $A_1\to A_2$, if the following conditions are satisfied:
\begin{enumerate}
    \item $A_1 \subset A_2$ and $A_2 \setminus A_1 = \{s_j\}$ for some strand $s_j\in S(D)$;
    \item $s_j$ is adjacent to $s_i$ at some crossing $c$ in $D$ with over-strand $s_k$, where $s_i, s_k\in A_1$.
\end{enumerate}
The move $A_1\to A_2$ reflects the fact that if a subgroup  $H\subset \pi_1(S^3-K, x_0)$ contains the Wirtinger generators corresponding to all strands in $A_1$, then $H$ also contains the generator corresponding to $s_j$; this is seen by applying the Wirtinger relation at $c$. 

We say $D$ is \textit{$k$-colorable} if\footnote{We have slightly simplified the original definition of $k$-colorability, which makes use of $k$ different colors. Multiple colors are needed in the proof of the Main Theorem of~\cite{blair2020wirtinger}, but they are of no help to us here. The modified definition has no effect on the value of $\omega(D)$.} there exists a subset $A_0$ of $S(D)$  
with $k$ elements and a sequence of $n - k$ coloring moves $A_0\to A_1\to\;\dots\;\to A_{n - k}$ on $D$ such that $A_{n-k}=S(D)$.  That is, after performing the sequence of coloring moves, every strand in $D$ is colored. It follows that the meridians of the strands in $A_0$ generate the link group via iterated application of the Wirtinger relations in $D$.  We refer to the elements of $A_0$ as the seed strands of the coloring sequence or, simply, the {seeds}. The smallest integer $k$ such that $D$ is $k$-colorable is the Wirtinger number of $D$, denoted $\omega(D)$. 

It is easy to come up with examples which demonstrate that $\omega(D)$ depends on the choice of diagram so is not a link invariant. In fact, the Wirtinger number can be arbitrarily large for sufficiently complicated diagrams of the unknot~\cite{blair2019incompatibility}. We are naturally more interested in minimizing the value of $\omega(D)$ over all diagrams $D$ of a given link $L$ since, by definition $$\omega(D)\geq \mu(L).$$

\begin{definition}
\label{omega}  Let $L\subset S^3$ be a link. The {\it Wirtinger number} of $L$, denoted $\omega(L)$, is the minimal value of $\omega(D)$ over all diagrams $D$ of $L$.	
\end{definition}

It is straight-forward to see that $\omega(L)$ satisfies the inequalities $$\beta(L)\geq\omega(L)\geq\mu(L).$$ 
In fact, the first inequality is never strict.

\begin{theorem}\cite[Theorem 1.3]{blair2020wirtinger}
\label{w=b} 
 Let $L\subset S^3$ be a link. The {\it Wirtinger number} and the bridge number of $L$ are equal.
 \end{theorem}

Therefore, given a diagram $D$ of a link $L$, we have

\begin{equation} \label{bwm}
	\omega(D)\geq\omega(L)=\beta(L)\geq\mu(L).
\end{equation}

We will prove Theorems~\ref{twisted} and~\ref{arbs} by using Coxeter quotients to show that, for links covered by these theorems, $\omega(D)$ is also a lower bound for the meridional rank. 

\section{Meridional rank conjecture for twisted links}

Let $L$ be a twisted link with diagram $D$ bounding a twisted surface $F$. In order to find the desired bounds on $\mu(L)$ and $\beta(L)$, it proves useful to retract the spanning surface $F$ to a graph as in~\cite{brunner1992geometric}. Since the boundary of a disk in $F$ contains multiple arcs in the knot diagram which represent different Wirtinger generators, it is convenient that the vertices of the resulting graphs be disks of non-zero radius, rather than points.  
We therefore work with fat-vertex graphs, which are planar graphs
 whose vertices are replaced by disjoint closed disks of small positive radius. These disks are the fat vertices. The boundary of each vertex is partitioned by the endpoints of incident edges into a finite collection of disjoint arcs. 
 A fat-vertex graph is {weighted} if an integer is assigned to each edge.

Given $L$, $D$ and $F$ as above, we obtain from $F$ a  weighted fat-vertex graph in the obvious way: view each disk of $F$ as a fat vertex and retract each twisted band of $F$ to its core edge, weighted by the number of (signed) half-twists of the band. We call this graph the fat-vertex graph associated to $D$ and, from here on, denote it by $\Gamma$. Also denote the dual weighted graph by $\Gamma^\ast$, where each edge of $\Gamma^\ast$ inherits the weight of the corresponding edge of $\Gamma$. For reasons that will become imminently apparent, we call the weighted graph $\Gamma^\ast$ the Coxeter graph associated to $D$. We suppress the choice of checkerboard coloring in this terminology and, in the case where $D$ is a twisted diagram, we are of course using the checkerboard coloring which detects this property. 

The surface $F$ is twisted if and only if all weights of $\Gamma$ are at least 2 in absolute value and  $\Gamma^\ast$ is a simple graph;  $L$ is then a twisted link. Under this assumption, Brunner~\cite{brunner1992geometric} shows that $\pi_1(S^3\backslash L)$ surjects to the Coxeter group $C(\Gamma^\ast)$ defined by the weighted graph $\Gamma^\ast$. Thereby, meridians of the boundaries of fat vertices are mapped to a generating set of reflections in $C(\Gamma^\ast)$. 
 Applying Proposition~\ref{lowerbound}, we conclude that the meridional rank of $L$ is bounded below by the rank of the Coxeter group determined by the graph $\Gamma^\ast$. 
This proves:
  
  \begin{lemma}\label{lower-bound}
  	Let $L$ be a twisted link with associated Coxeter graph $\Gamma^\ast$. The merdional rank of $L$ is bounded below by the number of vertices in $\Gamma^\ast$. 
  \end{lemma}
  
The next proposition, established later in this section, allows us to prove Theorem~\ref{twisted}.
  
  \begin{proposition}\label{upper-bound}
  	Let $L$ be a twisted link with associated Coxeter graph $\Gamma^\ast$. The bridge number of $L$ is bounded above by the number of vertices in $\Gamma^\ast$.  
  \end{proposition}

\begin{proof}[Proof of Theorem~\ref{twisted}] 
	Let $L$ be a twisted link and let $\Gamma^\ast$ be the Coxeter graph associated to a twisted diagram of $L$. Denote by $v$ the number of vertices of $\Gamma^\ast$. Combining Lemma~\ref{lower-bound} and Proposition~\ref{upper-bound}, we obtain $$v\geq \beta(L)\geq \mu(L)\geq v.$$
	That is, the meridional rank conjecture holds for twisted link and the bridge number of $L$ is equal to the number of vertices in $\Gamma^\ast$ or, equivalently, to the number of regions in the planar complement of a twisted surface for $L$. 
\end{proof}

In light of Theorem~\ref{twisted}, it is natural to ask which knots admit twisted diagrams. Prime twisted knots with at least 2 twist regions and at least 7 half-twists per region are hypebolic~\cite{futer2008dehn}. By results of Lackenby~\cite{L04}, the volume of such a knot is also bounded above by a constant of the bridge number. Hence, hyperbolic knots with high volume yet small bridge number are not covered by our theorem. However, our methods do allow us to establish the meridional rank conjecture for certain hyperbolic knots of fixed bridge number and arbitrarily high volume, e.g. two-bridge knots, as mentioned in the last paragraph of Section~\ref{lower}. More generally, Theorem~\ref{twisted} extends to a large class of links obtained from twisted links by replacing twist regions by rational tangles. For example, in Figure~\ref{pretzelknot} we can replace the last twist region by a rational tangle, say the one found at the very right of Figure~\ref{arborescentlink}, retaining the $a, c$ labels. This would preserve both the upper and lower bounds we found. The analogous construction can be performed in many situations, extending the proof of the meridional rank conjecture to the resulting knots. However, in general, replacing a twist region by a rational tangle does not preserve the Wirtinger number of a diagram.

\subsection{Proof of Proposition~\ref{upper-bound}}\label{sec-prop1}

For the remainder of this section, assume that $L$ is a twisted link with twisted diagram $D$. Denote the fat-vertex graph  and Coxeter graph associated to $D$ by $\Gamma$ and $\Gamma^\ast$, respectively. We will obtain the desired bound on $\beta(L)$ by the technique recalled in Section~\ref{upper}. It will be useful to be able to perform coloring moves not only on $D$ but directly on $\Gamma$.

\begin{definition}
	Let $\Gamma$ be a fat-vertex graph and denote by $E$ the set of edges of $\Gamma$. A {\it segment} of $\Gamma$ is either an element of $E$ or a connected arc contained in $\partial(v)\backslash\{\partial e| e\in E\}$, where $v$ is a fat vertex. 
\end{definition}

Denote by $S$ the set of segments of a fat-vertex graph $\Gamma$. We say that  $\Gamma$ is {\it partially colored} if we have fixed a function $f: S \to \{0, 1\}$ such that $A:=\{s|f(s)=1\}\neq\emptyset$. As before, refer to $A$ as a {\it partial coloring} of $\Gamma$ and to the elements of $A$ as the {\it colored segments}. Given two subsets $A_1\subset A_2\subset S$ with $A_2\backslash A_1 =\{s\}$ a single segment, we allow a {\it coloring move} $A_1\to A_2$ if one of the following holds:

\begin{enumerate}
	\item $s$ is an edge of $\Gamma$ and both segments adjacent to the same vertex of $s$ are in $A_1$.
	\item $s$ is an arc in the boundary of a fat vertex of  $\Gamma$ and $s$ is incident to an edge in~$A_1$.
\end{enumerate}

Case (1) in which a coloring move is allowed on $\Gamma$ is motivated by the following observation. Let $\Gamma$ be a fat-vertex graph obtained from a link diagram and spanning surface. An edge $e$ of $\Gamma$ denotes a twist region in the link diagram. The meridians of the two arcs incident to the same vertex of $e$ generate all meridians of strands contained in the corresponding twist region, via iterated Wirtinger relations, compare Figure~\ref{coloredbraid}. Case (2) is motivated by the fact that the meridians of arcs in a twist region generate the meridians of arcs incident to a twist region.  

Given a fat-vertex graph $\Gamma$, denote the set of its segments by $S(\Gamma)$, the set of its edges by $e(\Gamma)$ and the number of elements in $S(\Gamma)$ by $m$. We say $\Gamma$ is $k$-colorable if there exists a $k$-element subset $A_0$ of $S(\Gamma)\backslash e(\Gamma)$ and a sequence of $m - k$ coloring moves $A_0\to A_1\to\;\dots\;\to A_{m - k}$ on $\Gamma$ as defined above such that $A_{m- k}=S(D)$, that is, at the end of the coloring process every segment of $\Gamma$ is colored.  We refer to the elements of $A_0$ as the \textit{seed segments} or {\it seeds}. When  $\Gamma$ is the fat-vertex graph associated to a link diagram $D$, the seed segments correspond to meridional elements of $L$ that generate the group $\pi_1(S^3\backslash L)$ via iterated application of the Wirtinger relations in $D$. The minimum value of $k$ such that $\Gamma$ is $k$-colorable is the {\it Wirtinger number} of $\Gamma$, denoted $\omega(\Gamma)$. The following is immediate.

\begin{lemma} \label{w(g)}
	Let $D$ be a link diagram and $\Gamma$ its  associated fat-vertex graph. The inequality $\omega(\Gamma)\geq \omega(D)$ holds.
\end{lemma}

To complete the proof of Proposition~\ref{upper-bound}, we need one last ingredient, namely that a fat-vertex graph $\Gamma$ can be colored starting from as many seed segments as the number of vertices in $\Gamma^\ast$.  

\begin{lemma}\label{v*>w}
	Let $\Gamma$ be a connected fat-vertex graph associated to a reduced link diagram $D$. The Wirtinger number of $\Gamma$ is bounded above by the number of vertices in the dual graph~$\Gamma^\ast$.
\end{lemma}

But first: 

\begin{lemma}\label{theta}
	 Let $\Gamma$ be a connected finite plane graph with no loops, no separating vertices and no separating edges. Then either $\Gamma$ is a cycle or it contains a subgraph $T$ (not necessarily an induced one) that is homeomorphic to the theta graph, that is, the graph with two vertices connected by three parallel edges.
\end{lemma}

\begin{proof}
	This time we denote by $\Gamma^\ast$ the plane dual graph of $\Gamma$ where we omit the vertex corresponding to the unbounded region. The assumptions on $\Gamma$ imply that $\Gamma^\ast$ is connected. If $\Gamma^\ast$ is a single point, the absence of separating vertices implies that $\Gamma$ is a cycle. When $\Gamma^\ast$ has at least two vertices, $\Gamma$ contains two adjacent plane regions sharing one or several edges. The union of all the edges adjacent to these two regions contains an embedded theta graph.
\end{proof}

\begin{proof}[Proof of Lemma~\ref{v*>w}]

Since $D$ is reduced, it contains no nugatory crossings. Therefore, the graph $\Gamma$ has no leaves and no disconnecting edges. Indeed, a leaf in a fat-vertex graph corresponds to a region in a link diagram which can be removed by a sequence of Reidemeister I moves. Similarly, disconnecting vertices and edges correspond to nugatory crossings and connected sums. Nugatory crossings are not allowed in reduced diagrams. Secondly, it is enough to consider twisted diagrams of links whose components are prime, since both bridge number and Coxeter rank satisfy suitable additivity properties.
Given a connected sum of links $L=L_1\#L_2$, the inequality $$\beta(L_1)+\beta(L_2)-1\geq \beta(L)$$ is immediate. It is in fact equality~\cite{Schu54, doll1992generalized}, though we do not rely on this result. In addition, if the graph $\Gamma$ is obtained by identifying a vertex in $\Gamma_1$ with one in $\Gamma_2$, a vertex count gives the following relation among the ranks of the corresponding Coxeter groups: $$r(C(\Gamma))=r(C(\Gamma_1))+r(C(\Gamma_2))-1.$$  Therefore, if we denote the twisted links determined by these graphs by $L:=L(\Gamma)$, $L_1:=L(\Gamma_1)$ and $L_2:=L(\Gamma_2)$, by Proposition~\ref{lowerbound}, these links satisfy $$\mu(L)\geq \mu(L_1))+ \mu(L_2)-1.$$
Hence,  if the equality $\beta =\mu$ holds for each of $L_1$ and $L_2$, it also holds for the connected sum $L_1\#L_2$, seen as follows: $$\beta(L_1)+\beta(L_2)-1\geq \beta(L)\geq \mu(L)\geq \mu(L_1)+\mu(L_2)-1 =\beta(L_1)+\beta(L_2)-1.$$ We will thus assume that $\Gamma$ has no disconnecting vertices.

 In sum, we may prove the Lemma by an induction argument on connected fat-vertex graphs which are 1-connected, 1-edge-connected and have no vertices of valency one. Denote the set of such graphs by $\mathfrak{G}$. We will show that any element of $\mathfrak{G}$ can be obtained from the graph $G_0$ containing a single fat-vertex and no edges by a finite sequence of the following operations:

\begin{enumerate}
	\item  adding a self-loop to an existing vertex; 
	\item  subdividing an edge into 2 edges;
	\item adding an edge between two existing vertices.
\end{enumerate}

To verify that these operations suffice to construct all elements of $\mathfrak{G}$ from $G_0$, define the complexity of a graph $G\in\mathfrak{G}$ to be the integer $|v(G)| + |e(G)|$, the total number of edges and fat-vertices in $G$. Note that each of the operations (1)-(3) increases this complexity. To see that any $G\in\mathfrak{G}$ can be constructed from $G_0$ by a finite sequence of these operations, we show that given $G\in\mathfrak{G}, G\neq G_0$, one can undo one of the operations  (1)-(3) and remain within $\mathfrak{G}$. 

If $G\in\mathfrak{G}$ with $|v(G)| + |e(G)|>1$ has a loop $l$, we undo operation (1) on this loop. The resulting graph, $G'$, is connected. To see that $G'$ has no leaves, we only need to consider the vertex $v$ incident to the loop $l$, since no other vertex changes valency due to the removal of $l$. Suppose that $v$ is a leaf in $G'$. This implies that $v$ had valency three in $G$ and was incident only to $l$ and one other edge $e$. It follows that $e$ was a disconnecting edge in $G$, a contradiction. Note also that removing a self-loop can not create disconnecting edges or vertices. Therefore, $G'\in\mathfrak{G}$. 

Now suppose that $G$ does not have a loop but has a degree-2 vertex $v$, incident to edges $e_1$ and $e_2$. In this case, undo operation (2) at $v$, producing a new edge $e$. The resulting graph, $G'$, is connected. It has no leaves or disconnecting vertices since $G$ had none. If $e$ is a disconnecting edge in $G'$, then $e_1$ was one in $G$. If removing some other edge in $G'$ would disconnect the graph, then removing the same edge would disconnect $G$. Therefore, $G'\in\mathfrak{G}$. 

Finally, assume that $G$ contains neither a loop nor a degree-two vertex. Since $G\in\mathfrak{G}$ and has no loops, it satisfies the hypotheses of Lemma~\ref{theta}. Since $G$ has no degree-two vertex, it is not a cycle. Therefore, $G$ contains an embedded theta graph $\theta$ consisting of three paths $\gamma_1,\gamma_2,\gamma_3$ connecting two vertices $v,w$ of $G$. Choose three edges $e_1,e_2,e_3$ of $G$ with $e_i \subset \gamma_i$. Let $G'$ be the graph obtained from $G$ by removing the edge $e_1$ and all edges connected to $e_1$ by a chain of vertices of valency two, thereby undoing operation (3), after possibly undoing multiple operations of type (2). It is clear that $G'$ is connected and has no leaves. Furthermore, if some vertex in $G'$ is disconnecting, the same vertex is seen to be disconnecting in $G$.  What requires a check is that $G'$ has no disconnecting edges. Assume there is an edge $e$ in $G'$ such that  $G'\backslash e$ is disconnected. Let $a$ and $b$ be two vertices in different components of $G'\backslash e$.  By assumption, $e$ is not a disconnecting edge in $G$, so there is a path $\gamma$ in $G$ such that $\gamma$ connects $a$ to $b$ and does not contain $e$. If $\gamma$ does not meet the chain of edges $G \setminus G'$, then $\gamma$ is also a path in $G'$, contradiction. If $\gamma$ meets that chain, then we may replace each connected component of $\gamma \cap (G \setminus G')$ by a point (if it does not traverse the entire chain) or by an arc in the subgraph $\theta$ passing through $\gamma_2$ or $\gamma_3$, avoiding the chain $G \setminus G'$. We thus obtain a path in $G'$ connecting $a$ to $b$, contradiction.  This shows that $G'\in\mathfrak{G}$.

Now let $\Gamma\in\mathfrak{G}$ be as in the statement of the Lemma and denote by $k$ the number of vertices in $\Gamma^\ast$. We will show inductively that $\Gamma$ is $k$-colorable. 

Let $\Gamma=G_0$, the graph consisting of a single fat-vertex. In this case,  $\Gamma^\ast$ has one vertex, so $k=1$. It is clear that $\Gamma$ is 1-colorable: choose the only segment of $\Gamma$ as the seed. 

Now assume  $\Gamma_1\in\mathfrak{G}$ is $k_1$-colorable, where $k_1$ is the number of vertices in $\Gamma_1^\ast$, and let  $\Gamma\in \mathfrak{G}$ be obtained from $\Gamma_1\in\mathfrak{G}$ by one of the operations (1)-(3).  We will show that $\Gamma$ is $k$-colorable, where $k$ denotes the number of vertices in $\Gamma^\ast$. 

By assumption, there exists a coloring sequence $A_0\to A_1\to \dots\to A_n$ for $\Gamma_1$, where each $A_i$ is a set of segments in $\Gamma'$. Moreover, $A_0$ has $k_1$ elements, and $A_{i+1}\backslash A_i$ contains exactly one element. Order the set of segments $S(\Gamma_1)$ as $s_1, s_2, \dots, s_{k_1}, s_{k_1+1}, \dots s_{k_1+n}$, where $\{s_1, \dots, s_{k_1}\}$ are the elements of $A_0$, taken in any order, and for $m>k_1$, $s_m$ is the segment colored when the $(m-k_1)$-th coloring move is performed. At the risk of minor ambiguity, we will call this sequence of segments a coloring sequence for $\Gamma_1$ as well, and we will use it to produce the desired coloring sequence for $\Gamma$. 

{\it Case A.} Suppose that $\Gamma$ is obtained from $\Gamma_1$ by subdividing an edge $e\in S(\Gamma_1)$ into two edges $e_1, e_2\in S(\Gamma)$, both incident to a degree-two fat-vertex $v$ in $\Gamma$. By construction, the boundary of $v$ contains two segments; denote them $a$ and $b$. We see that $$S(\Gamma)=\{S(\Gamma_1)\backslash\{e\} \}\cup\{e_1, e_2, a, b\}.$$

We will produce a coloring sequence for $\Gamma$ from the given coloring sequence $s_1,  \dots, s_{k_1},$ $\dots, s_{k_1+n}$ for $\Gamma_1$. Since $\{s_1, \dots, s_{k_1}\}$ are seed segments and $e$ is an edge, $e$ appears after $s_{k_1}$ in the sequence. Denote the position of $e$ by $r>k_1$ and rewrite:   $$s_1,  \dots, s_{k_1},\dots, s_{r-1}, e, s_{r+1}, \dots, s_{k_1+n}.$$ 
A valid coloring sequence for $\Gamma$ is then $$s_1,  \dots, s_{k_1},\dots, s_{r-1}, e_1, a, b, e_2, s_{r+1}, \dots, s_{k_1+n}.$$
Here, the seeds are $\{s_1, \dots, s_{k_1}\}$ and we have chosen notation so that $e_1$ is the edge incident to segments contained in the boundary of fat-vertex which are among $\{s_1,  \dots, s_{r-1}\}$. It is clear that either $e_1$ or $e_2$ has this property since a coloring move was performed on $\Gamma_1$, coloring $e$ from one of the two fat-vertices it is incident to. Once $e_1$ is colored, $a$ and $b$ can be colored in any order; this, in turn, allows us to color $e_2$. All remaining coloring moves in this sequence are valid because so were the analogous coloring moves on $\Gamma_1$. 

Thus, we have exhibited a coloring sequence for $\Gamma$ starting with $k$ seeds, where by assumption $k$ is the number of vertices in $\Gamma_1^\ast$. Since $\Gamma$ was obtained from $\Gamma_1$ by adding an edge and a vertex, Euler characteristic shows that $\Gamma^\ast$ has $k$ vertices as well. Therefore, the coloring sequence produced has the desired number of seeds.

{\it Case B.} Suppose that $\Gamma$ is obtained from $\Gamma_1$ by adding an edge $e$. Denote by $a$ and $b$ the segments in $\Gamma_1$ containing the endpoints of the new edge $e$. We consider operations (2) and (3) simultaneously, that is, we allow for the possibility that $a$ and $b$ are the same arc. 
If $a$ and $b$ are distinct, let $a_1, a_2$ denote the segments in $\Gamma$ into which $e$ subdivides $a$ and, similarly, let $b_1, b_2$ be the segments in $\Gamma$ into which $e$ subdivides $b$. 
We have $$S(\Gamma)=\{S(\Gamma_1)\backslash\{a, b\} \}\cup\{a_1, a_2, e, b_1, b_2\}.$$
In the case where $a$ and $b$ are the same arc, the endpoints of $e$ subdivide $a$ into arcs $a_1$, $a_2$, and $b_2$ and the new edges are $\{e, a_1, a_2, b_2\}$. In what follows, assume that $a_1$ and $a_2$ are incident to the same vertex of $e$. 

Denote by $s_1,  \dots, s_{k_1}, \dots, s_m$ a coloring sequence for $\Gamma_1$ in which  $\{s_1,  \dots, s_{k_1}\}$ are the seeds. From this, we will construct the desired coloring sequence for $\Gamma$. Since $\Gamma$ is obtained from $\Gamma_1$ by adding an edge, by Euler characteristic we see that $\Gamma^\ast$ has $k_{1}+1$ vertices, so we can use an extra seed when coloring $\Gamma$.

Since $a$ and $b$ are segments in $\Gamma_1$, they appear in the coloring sequence as some $s_i, s_j$ where, without loss of generality, $i\leq j$. 
 We can rewrite the sequence as 
$$s_1, \dots, s_{i-1}, a, s_{i+1},\dots, s_{j-1}, b, s_{j+1}, \dots, s_m.$$ 
Here, $a$ is either a seed (if $i\leq k_1$) or becomes colored via a coloring move (if $i>k_1$). We consider both cases.

{\it Case B1.} If $a$ is a seed, the labels $a_1$ and $a_2$ can be assigned arbitrarily to the segments into which the new edge $e$ subdivides $a$.  A possible coloring sequence for $\Gamma$ with $k_1+1$ seeds will be:
$$s_1, \dots, s_{i-1}, a_1, a_2, s_{i+1}, \dots, s_{k_1},  e, b_1, b_2, s_{k_{1}+1},\dots, \hat{s_{j}} \dots, s_m.$$ 
(If $a=b$, omit $b_1$.) 

The seeds here are $\{s_1, \dots, s_{i-1}, a_1, a_2, s_{i+1}, \dots, s_{k_1}\}$. Let us check that the sequence is valid, that is, each segment which appears after $s_{k_1}$ is colored by a permitted coloring move. The edge $e$ can be colored since one of its vertices is incident to the two arcs $a_1, a_2$, both of which precede $e$ in the sequence. Next, $b_1, b_2$ can be colored since they are incident to $e$.  All remaining coloring moves are valid because so were the analogous coloring moves on $\Gamma_1$. 

{\it Case B2.} If $a$ was colored via a coloring move, it follows that there is an edge $s_l$, $l<i$, which shares an endpoint with $a$. 
To write down a coloring sequence in this case, say $a_2$ is the segment in $\Gamma$ which shares an endpoint with $s_l$ after subdivision. With this notation, a possible coloring sequence for $\Gamma$ is 
$$a_1, s_1, \dots, s_{k_1}, \dots, s_{i-1}, a_2, e, b_1, b_2, s_{i+1},\dots, \hat{s_{j}}, \dots, s_m$$ 
 (Again, if $a=b$, omit $b_1$.) 

The $k_{1}+1$ seeds here are $\{a_1, s_1, \dots, s_{k_1}\}$. Again, we need to verify that the coloring moves performed are valid. By assumption, $a_2$ is  incident to the colored edge $s_l$, $l<i$, so $a_2$ can be colored. Then, as in the previous case, the edge $e$ can be colored since it is incident to $a_1$, and $b_1, b_2$ can be colored since they are incident to $e$. The remaining coloring moves are valid because so were the analogous coloring moves on $\Gamma_1$. 

Therefore, any $\Gamma\in\mathfrak{G}$ can be colored from as many seeds as the number of vertices in~$\Gamma^\ast$.
\end{proof}

Proposition~\ref{upper-bound} is now a simple consequence of the above results. Recall that $L$ denotes a link with reduced diagram $D$, associated fat-vertex graph $\Gamma$ and Coxeter graph $\Gamma^\ast$. Let $v(\Gamma^\ast)$ be the number of vertices in this graph. Combining Equation~\ref{bwm} with Lemmas~\ref{w(g)} and~\ref{v*>w} gives
 $$v(\Gamma^\ast)\geq\omega(\Gamma)\geq\omega(D)\geq \beta(L).$$

\section{Meridional rank conjecture for bipartite arborescent links}
\label{trees}

The proof of Theorem~\ref{arbs} is again based on the existence of Coxeter quotients whose rank matches the Wirtinger number in a suitable link diagram. We start by deriving an upper bound on the bridge number of general arborescent links.

\begin{proposition} \label{leaves}
Let $L(T)$ be an arborescent link associated with a plane tree $T$ with arbitrary weights. Then the bridge number of $L(T)$ is bounded above by the number of leaves of $T$. 
\end{proposition}

In view of Theorem~\ref{arbs}, we may expect the bridge number of arborescent links to equal the number of leaves in the underlying tree. However, this is false. The knot associated with the even weight tree with four leaves shown in Figure~\ref{badtree} turns out to be a three-bridge knot. 

\begin{figure}[h]
\begin{center}
\includegraphics[scale=1.8]{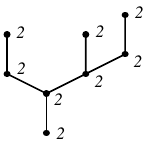}
\caption{Non-bipartite tree with even weights defining a 3-bridge knot. \label{badtree}}
\end{center}
\end{figure}

Incidentally, this is where our condition on trees originates from. The bipartite structure of trees turns out to be essential in proving the existence of Coxeter quotients whose rank is equal to the number of leaves.

\begin{proof}[Proof of Proposition~\ref{leaves}]
We will prove that the Wirtinger number of the natural arborescent diagram of $L(T)$ is bounded above by the number $n$ of leaves of $T$, by induction on $n$. The base case -- two leaves -- is easy, since the family of links associated with even weight trees with two leaves coincides with the class of two-bridge links~\cite{conway1970enumeration}.

Let $T$ be a plane tree with $n$ leaves. While there is no canonical link diagram associated to $T$, there is a natural construction depending on an initial choice of branching point in $T$.  We highlight an essential property of this construction.  A branch in $T$ is a chain of edges connecting a leaf to the first vertex of valency at least 3. The diagram is then built from a single twisted band by successively adding $n-1$ rational tangles, one for each branch in $T$. The convention can be chosen so that the ``rightmost" branch in the tree corresponds to a rational tangle, labeled $B$ in Figure~\ref{bottomtangle}, located in the bottom right corner of the diagram. This rightmost branch is the one we add in the inductive step, thereby increasing the number of leaves in $T$ by one and simultaneously adding a rational tangle to a diagram assumed colorable with $n-1$ seeds.

We now show that the diagram of $L(T)$ is colorable with $n$ seeds, one of which, say $s_0$, is at the bottom right of the diagram.
\begin{figure}[h]
\begin{center}
\includegraphics[scale=1.0]{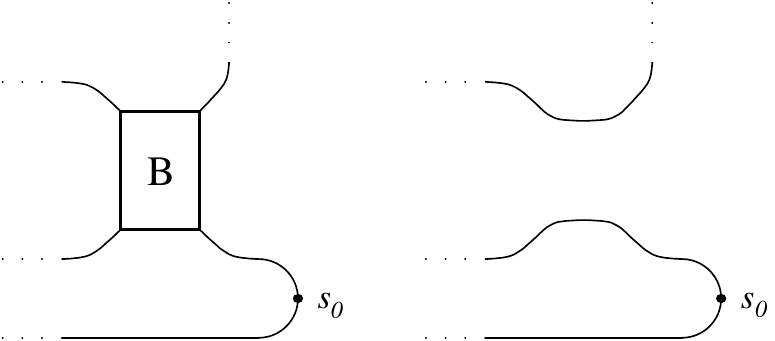}
\caption{Bottom tangle~$B$ and seed \label{bottomtangle}}
\end{center}
\end{figure}

The key observation is that we can place a new seed $s_1$ inside the tangle $B$, so that the partial coloring defined by $s_0$ and $s_1$ propagates to the four outgoing strands of $B$. This is illustrated in Figure~\ref{tangleseed}, for a rational tangle of even and odd length (the sign and number of crossings are irrelevant there). But now we are done, since we can remove the tangle $B$, as shown on the right of Figure~\ref{bottomtangle}, which amounts to removing one branch of the tree $T$, and use the induction hypothesis for trees with $n-1$ leaves. This procedure yields one seed per rational tangle, in addition to the initial seed~$s_0$, as seen in Figure~\ref{arborescentlink}.
\begin{figure}[h]
\begin{center}
\includegraphics[scale=1.0]{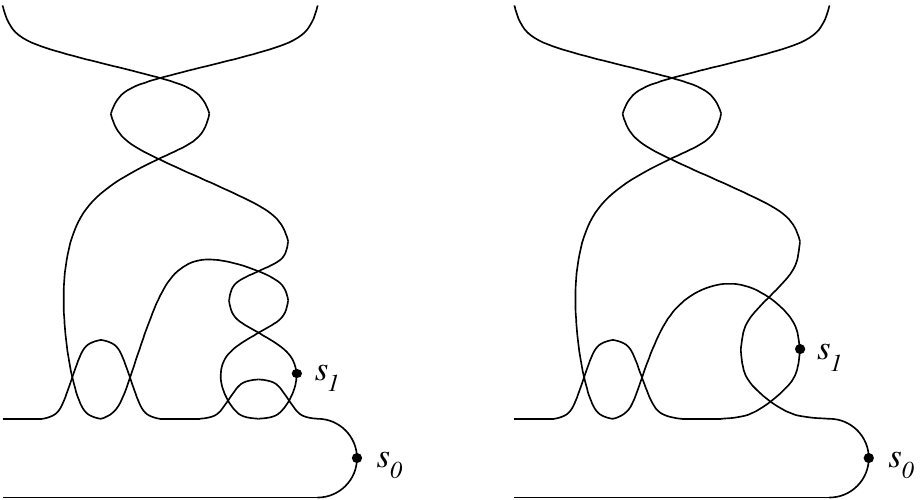}
\caption{Coloring rational tangles \label{tangleseed}}
\end{center}
\end{figure}
\end{proof}

The following proposition settles the proof of Theorem~\ref{arbs}; together with Proposition~\ref{leaves}, it provides the desired equality between the bridge number of the link $L(T)$, its meridional rank, and the number of leaves of the underlying even weight bipartite tree $T$.

\begin{proposition} \label{coxeterquotient}
Let $L(T)$ be an arborescent link associated with a plane bipartite tree $T$ with non-zero even weights. Then the fundamental group of $S^3\backslash L(T)$ admits a Coxeter group quotient whose rank is equal to the number of leaves of $T$. In particular, the meridional rank of the link $L(T)$ is bounded below by the number of leaves of $T$.
\end{proposition}

\begin{proof}
The proof is again by induction on the number of leaves of the tree $T$.
Here the case of two leaves is less obvious, since it amounts to proving that the fundamental group of non-trivial two-bridge links admits a Coxeter group quotient of rank two. This is just what we did in the last paragraph of Section~\ref{lower}. To be more precise, we only dealt with the case of two-bridge knots there. The case of two component two-bridge links is trivial, since their fundamental group admits $\mathbb{Z}^2$ as a quotient, thus the rank two Coxeter group~$(\mathbb{Z}/2\mathbb{Z})^2$.

Let $T$ be a plane bipartite tree with non-zero even weights and $n$ leaves. Every vertex of $T$ correponds to a twist region in the natural diagram of $L(T)$. These come in two versions, horizontal and vertical, which alternate between adjacent vertices, as in Figure~\ref{arborescentlink}. The bipartite condition on the tree $T$ means that all vertices of valency at least three correspond to the same type of twist region, say the horizontal one. We will construct a Coxeter quotient $G$ of rank $n$ with the following additional property: For all vertices of valency at least three, the meridians of the corresponding twist region are sent to the same reflection in a generating set for $G$.

An example of such a quotient is defined by the labeling in Figure~\ref{arborescentlink}. The quotient group there is the Coxeter group generated by the four reflections $a,b,c,d$ satisfying the Coxeter relations
$$(ab)^4=1,(ac)^3=1,(bc)^2=1,(bd)^4=1,(cd)^5=1.$$
In that diagram, there are two twist regions carrying a single label ($a$ and $b$); they correspond to the two vertices of valency three.

For the inductive step, there are two cases to consider:

Case 1. A branch is added to the tree $T$, at a vertex $v$ of valency at least three. Suppose the arborescent link diagram associated with $T$ admits a labeling by elements of a rank $n$ Coxeter group defining a rank $n$ Coxeter quotient of the fundamental group. Additionally, we suppose that the twist region of the vertex $v$ carries a single label, say $a$, and every other branch of $T$ corresponds to a rational tangle with a rank $2$ Coxeter labeling. Moreover, by induction we can assume that the labels incident to the boundaries of these rational tangles are all taken from the generating set of the rank $n$ Coxeter group. Then, we can add a branch at $v$, i.e. we can add a rational tangle $X$ with a rank $2$ labeling at the twist region of $v$, by introducing a label $x$, as shown in Figure~\ref{addbranch1}. Here, $x$ denotes a new generator, which, together with the previous $n$ generators, defines a rank $n+1$ Coxeter quotient. 
\begin{figure}[h]
\begin{center}
\includegraphics[scale=1.0]{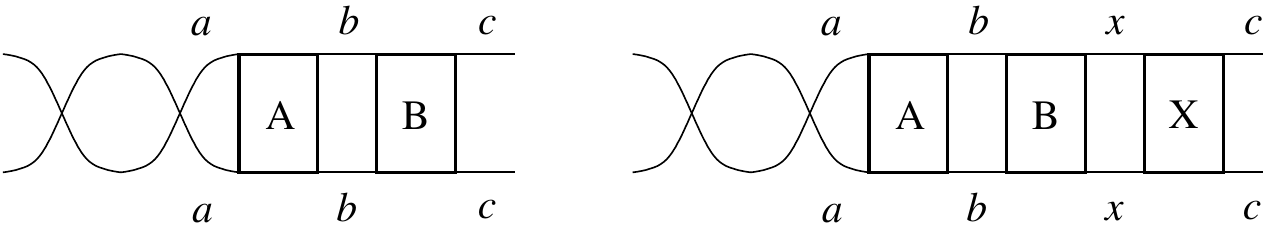}
\caption{Inserting a rational tangle, case~1. \label{addbranch1}}
\end{center}
\end{figure}
The reflection $x$ satisfies Coxeter relations with the two neighbouring generators ($b$ and $c$, in the figure). These are determined by the rational tangle $X$ and its neighbour ($B$, in the figure).

Case 2. A branch is added to the tree $T$, at a vertex $v$ of valency two. 

Again, we suppose the arborescent link diagram associated with $T$ admits a labeling by elements of a rank $n$ Coxeter group defining a rank $n$ Coxeter quotient of the fundamental group. Additionally, we suppose that the twist region of the vertex $v$ carries a single label, say $a$, and every other branch of $T$ corresponds to a rational tangle with a rank $2$ Coxeter labeling. Moreover, we assume that the labels incident to the boundaries of these rational tangles are all taken from the generating set of the rank $n$ Coxeter group. However, this time we cannot suppose that the twist region of the vertex $v$ carries a single label. Rather, the twist region of the vertex $v$ is part of a rational tangle $B$, as shown at the top of Figure~\ref{addbranch2}.
\begin{figure}[h]
\begin{center}
\includegraphics[scale=1.0]{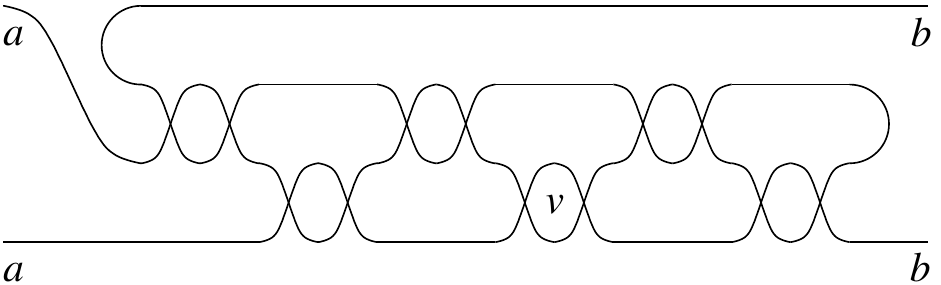}

\bigskip
\bigskip
\bigskip
\includegraphics[scale=1.0]{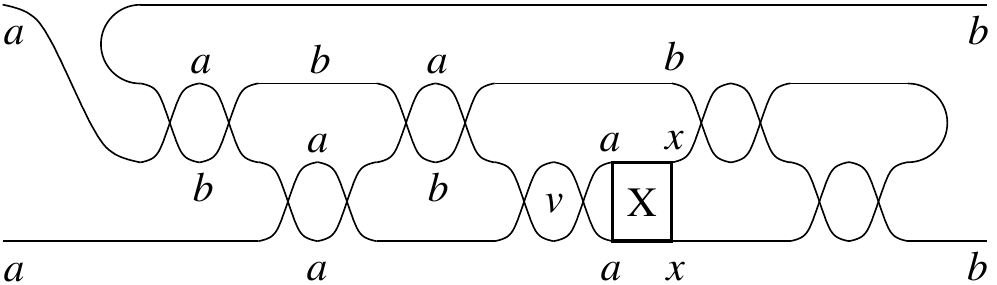}
\caption{Inserting a rational tangle, case~2. \label{addbranch2}}
\end{center}
\end{figure}

For illustration purposes, we chose a tangle with six twist regions, three of which are ``horizontal" (the ones on the bottom line). The labeling of the arborescent link diagram $L(T)$ associates Coxeter generators $a$ and $b$ to the four outgoing strands of the rank $2$ labeled rational tangle $B$, satisfying a Coxeter relation determined by $B$. Now we insert a rank $2$ labeled rational tangle $X$ at the twist region of the vertex $v$, and introduce a new Coxeter generator $x$, as shown in Figure~\ref{addbranch2}. The generator $x$ satisfies Coxeter relations with generators $a$ and $b$, determined by the rational tangle $X$, and the rational leftover tangle on the right of $X$. (For the explicit Coxeter relation, see again the last paragraph of Section~\ref{lower}.) Finally, the original Coxeter relation between $a$ and $b$ is replaced by $(ab)^2=1$, equivalently $ab=ba$. This is the only place where we use the condition that all weights are even, i.e. that all twist regions have an even number of crossings\footnote{When the tree has only one branching point, it is not necessary to restrict to even weights. In particular, our proof recovers the meridional rank conjecture for Montesinos links.}.

In both of the above cases, the number of leaves and the rank of the Coxeter quotient simultaneously increase by one. With this observation, we conclude the proofs of Proposition~\ref{coxeterquotient} and Theorem~\ref{arbs}.
\end{proof}

{\bf Acknowledgements.} Part of this work was completed while SB was visiting AK at the Max Planck Institute for Mathematics. We thank MPIM for its support and hospitality. RB and AK were also supported by NSF grants DMS-1821254 and DMS-1821257. 
We are grateful to Michel Boileau and Filip Misev for inspiring discussions, and to Curtis Bennett for many helpful conversations.

\bibliographystyle{amsplain}
\bibliography{wirtinger.bib}

\end{document}